\documentclass
{siamonline220329}

\usepackage{lipsum}
\usepackage{amsfonts}
\usepackage{graphicx}
\usepackage{epstopdf}
\usepackage{algorithmic}
\ifpdf
  \DeclareGraphicsExtensions{.eps,.pdf,.png,.jpg}
\else
  \DeclareGraphicsExtensions{.eps}
\fi

\usepackage{enumitem}
\setlist[enumerate]{leftmargin=.5in}
\setlist[itemize]{leftmargin=.5in}

\newsiamremark{remark}{Remark}
\newsiamremark{hypothesis}{Hypothesis}
\crefname{hypothesis}{Hypothesis}{Hypotheses}
\newsiamthm{claim}{Claim}

\headers{Area formula for spherical polygons via prequantization}{Albert Chern and Sadashige Ishida}
\title{Area formula for spherical polygons via prequantization 
}

\author{
Albert Chern\thanks{University of California San Diego
  (\email{alchern@ucsd.edu}).}
\and Sadashige Ishida\thanks{Institute of Science and Technology Austria 
  (\email{sadashige.ishida@ist.ac.at}).}
\funding{This project was funded in part by the European Research Council (ERC Consolidator Grant 101045083 \emph{CoDiNA}) and the National Science Foundation CAREER Award 2239062. Some figures in the article were generated by the software Houdini and its education license was provided by SideFX.}
}

\usepackage{amsopn}

\ifpdf
\hypersetup{
  pdftitle={Area formula for spherical polygons via prequantization},
  pdfauthor={Albert Chern and Sadashige Ishida}
}
\fi

\usepackage{mathtools}
\usepackage{xcolor}
\usepackage{nicefrac}

\theoremstyle{remark}

\numberwithin{equation}{section}

\usepackage{wrapfig}
\usepackage{graphicx} 

\newcommand{\abs}[1]{\lvert#1\rvert}

\usepackage[draft]{pdfcomment}

\usepackage{hyperref}
\hypersetup{
    colorlinks=true,
    linkcolor=black,
    filecolor=magenta,      
    urlcolor=magenta,
}

\usepackage[subrefformat=parens]{subcaption}
\newcommand{\stkouteq}[1]{\ifmmode\text{\sout{\ensuremath{#1}}}\else\sout{#1}\fi}
\newcommand{\sttext}[1]{\setstcolor{red}\st{#1}} 

\newcommand{\figref}[1]{Figure~\ref{#1}}

\newcommand{\secref}[1]{Section~\ref{#1}}

\newcommand{\thmref}[1]{Theorem~\ref{#1}}
\newcommand{\lemref}[1]{Lemma~\ref{#1}}

\newcommand{\propref}[1]{Proposition~\ref{#1}}
\newcommand{\corref}[1]{Corollary~\ref{#1}}

\def\cf{\emph{cf.}}
\def\ie{\emph{i.e.}}

\def\CC{\mathbb{C}}
\def\DD{\mathbb{D}}

\def\HH{\mathbb{H}}

\def\PP{\mathbb{P}}

\def\RR{\mathbb{R}}
\def\SS{\mathbb{S}}

\def\ZZ{\mathbb{Z}}

\def\bx{\mathbf{x}}

\usepackage{bbm}
\DeclareSymbolFont{bbold}{U}{bbold}{m}{n}
\DeclareSymbolFontAlphabet{\mathbbold}{bbold}
\newcommand{\ii}{\mkern1.5mu\mathbbold{i}\mkern1.5mu}
\newcommand{\jj}{\mkern1.5mu\mathbbold{j}\mkern1.5mu}
\newcommand{\kk}{\mkern1.5mu\mathbbold{k}\mkern1.5mu} 

\newcommand{\overbar}[1]{\mkern 1.5mu\overline{\mkern-1.5mu#1\mkern-1.5mu}\mkern 1.5mu}
\newcommand*{\conj}[1]{\overbar{#1}}

\renewcommand{\Im}{\operatorname{Im}}
\renewcommand{\Re}{\operatorname{Re}}

\def\det{\operatorname{det}}

\DeclareMathOperator{\SO}{SO}

\def\tt{\mathbbold{t}}

\DeclareMathOperator{\Area}{Area}

\DeclareMathOperator{\Dihedral}{Dihedral}

\begin{document}

\maketitle

\begin{abstract}
    We present a formula for the signed area of a spherical polygon via prequantization.  In contrast to the traditional formula based on the Gauss--Bonnet theorem that requires measuring angles, the new formula mimics  Green's theorem and is applicable to a wider range of degenerate spherical curves and polygons.
\end{abstract}



\section{Introduction}
\label{sec:Introduction}
A spherical polygon is a finite number of ordered points on \(\SS^2\) connected by geodesics. Computing the solid angle of the region enclosed by a spherical polygon is important in many subjects, which we briefly summarize at the end of this section.

For a given polygon \(\Gamma=(p_0, \ldots,p_{n-1})\), its surface area is often computed using the formula 
\begin{align}\label{eq:GB_formula}
\operatorname{Area}(\Gamma)=2\pi - \sum_i \vartheta_i, 
\end{align}
derived from the Gauss--Bonnet theorem.
This area formula involves the evaluation of the exterior angle \(\vartheta_i\) at each vertex 
\begin{align}\label{eq:exterior_angle1}
    \vartheta_i=\operatorname{sign}\left(\det (p_{i-1},p_i,p_{i+1}) \right)\arccos\left(\frac{p_{i-1}\times p_{i}}{\abs{p_{i-1}\times p_{i}}}\cdot\frac{p_i\times p_{i+1}}{|p_i\times p_{i+1}|}\right),\  
\end{align}
which is a function of three points.
However, this formula requires non-degeneracy assumptions about these points that render it unavailable or numerically unstable in certain situations. 
For example, two consecutive points of the polygon cannot lie on the same location as the exterior angle is undefined. 
We may consider removing such points from the polygon, but judging precisely whether two points are at the same location or just nearby locations is not always possible in numerical computation, especially when the points are obtained after some computational operations. 

On the other hand, an area formula for polygons in \(\RR^2\) does not have such limitations. 
For a polygon \(\Gamma_{\RR^2}=\left((x_0,y_0).\ldots,(x_{n-1},y_{n-1})\right)\), the area is 
\begin{equation}\label{eq:planerAreaFormula}
    %
    \operatorname{Area}(\Gamma_{\RR^2})=\sum_i \frac{1}{2}(x_{i+1}y_i-x_i y_{i+1}).
\end{equation}
\begin{wrapfigure}[12]{r}{120pt}
\centering
\includegraphics[width=120pt]{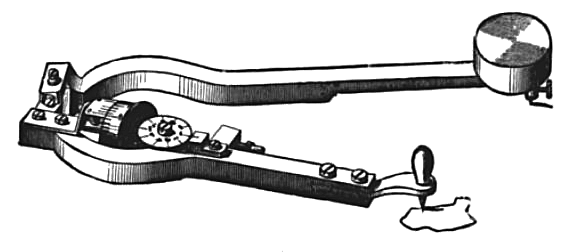}
\caption{A planimeter measures the area of a planar region by tracing and accumulating quantities along its perimeter.}
\label{fig:Planimeter}
\end{wrapfigure}
Unlike the classical formula for spherical polygons (\ref{eq:GB_formula}, \ref{eq:exterior_angle1}),
this planar area formula involves adding numerically stable edge quantities.
This formula can be derived by converting the area integral over the region into a line integral along the polygon curve using Green's formula. 
This is the same technique that enables 
\emph{planimeters} (\autoref{fig:Planimeter}).
The key that allows Green's formula is that the area form of \(\RR^2\) is an exact differential form.
Unfortunately, this strategy does not work directly for spherical polygons as the area form of \(\SS^2\) is not exact.

We circumvent this issue in this work via so-called \emph{prequantization}, which is a preliminary setup to transform a classical mechanical system into a quantum mechanical system \cite{bates_Weinstein1997lectures}. 
We lift the area form of the sphere onto a space where the resulting 2-form is exact. More precisely, we utilize a prequantum bundle, 
which is a principal circle bundle where the lifted 2-form is exact.
This gives rise to a version of Green's theorem that translates the area integral on the base manifold into a line integral along a lifted perimeter in the bundle.

By choosing a specific prequantum bundle over \(\SS^2\), we can obtain an explicit expression of the line integral. We use the Hopf fibration \(\pi\colon \SS^3 \to \SS^2\) with a specific connection 1-form and derive a formula for the  area of a spherical polygon. 
Unlike the classical formula \eqref{eq:GB_formula}, the new formula does not involve numerically unstable evaluation of angles of three consecutive vertices.  Instead, it only involves a sum of edge-wise measurement resembling \eqref{eq:planerAreaFormula}.  
We also recover the classical formula \eqref{eq:GB_formula} by choosing \(\SO(3)\) as the prequantum bundle with a specific lift of the polygon.
Finally, we present numerical examples comparing our formula and the classical formula. 

\paragraph{Relation to quantum mechanics}
Our formula for the spherical area via prequantization is similar to evaluating the \emph{Berry phase} for a quantum mechanical system.  In this analogy, the points on \(\SS^3\) (or any prequantum circle bundle) represent the possible quantum states of a spinor (or 2 qubits), while their projections on \(\SS^2\) are their classical representations on the \emph{Bloch sphere} \cite{bengtsson_zyczkowski_2006}.

\paragraph{Applications and computations of solid angles}

Computing the area, or solid angle, of a spherical region occurs in multiple disciplines.
They are either analytically computed by the angle-based formula (\ref{eq:GB_formula}, \ref{eq:exterior_angle1}) \cite{Bevis1987ComputingTA}, or approximated by 
point counting \cite{feng2006area,Gonzalez2010Fibonacci} and Monte-Carlo methods \cite{Arvo1995stochasticSampling}.
A standard routine in geological survey is to find the area of an irregularly shaped region on the nearly-spherical earth   \cite{Bevis1987ComputingTA,Gonzalez2010Fibonacci}.
In quantum mechanics, the solid angle of a spherical curve generated by a Brownian motion of a particle describes the phase of two-state quantum systems undergoing random evolution \cite{shapere1989geometric,Krishna_2000}. The so-called Majorana representation of polarized light also involves the solid angle of certain spherical quadrilaterals \cite{Hannay1998Majorana}.

Solid angles give rise to another geometric concept, \emph{the solid angle field of a space curve}.  Given a space curve in \(\RR^3\), the solid angle field is an $\SS^1$-valued function over $\RR^3$ whose value at each point $\bx\in\RR^3$ is half of the solid angle subtended by the closed curve at \(\bx\).
In contact geometry, a so-called \emph{open-book decomposition} can be constructed by the solid angle field, with the level sets of the \(\SS^1\)-valued solid angle field being the \emph{pages} and the space curve being the \emph{binder} \cite{Geiges_2008contact}.
Each page can also serve as a \emph{Seifert surface}, an oriented surface bordered by a given knot \cite{dangskul2016construction,Binysh_2018}.
Solid angle fields also play important roles in fluid dynamics and electrodynamics, as the gradients of the solid angle fields are the \emph{Biot-Savart fields} of space curves \cite[pp.118]{sommerfeld2013electrodynamics}, which represent the magnetic fields induced by electric currents, and the velocity fields corresponding to given vortex filaments.
In addition, solid angle fields are applied to constructing implicit representations for space curves for simulating vortex dynamics \cite{iwc2022implicit_filaments}, visualization of nematic dislinations in electromagnetism \cite{Binysh_2018}, and radiosity illumination in rendering
\cite{Ramamoorthi2022}.

\section{Area of a spherical polygon}
We begin with the problem of seeking a line integral formula for the areas of spherical polygons.
Next, we introduce the notion of prequantum bundles.  While we only use the essential properties of this notion in a self-contained manner, the readers may find backgrounds in principal bundles \cite{nakahara2003geometry,kobayashiNomizu1963foundations} and geometric quantization \cite{bates_Weinstein1997lectures,kostantBertram1970quantization}useful.
Finally, we derive a line integral formula for spherical areas via prequantization and apply it to spherical polygons.

  An oriented spherical polygon \(\Gamma\) is a finite cyclic ordered list of spherical points \(\Gamma=(p_0,\ldots,p_{n-1})\), \(p_i \in \SS^2\), \(i\in\ZZ_n=\ZZ/(n\ZZ)\).
 Each edge, \ie\@ each pair of adjacent points 
  \((p_i, p_{i+1})\), \(i\in\ZZ_n\) (including \((p_{n-1},p_0)\) using the modulo arithmetic of \(i\in\ZZ_n\)), is joined by the shortest connecting path on \(\SS^2\).  
  This edge path is a constant point when \(p_i=p_{i+1}\), and is otherwise a part of the great circle containing \(p_i,p_{i+1}\).   
  To ensure uniqueness of the shortest edge path, we assume that \(p_{i+1}\neq -p_i\) for each \(i\in\ZZ_n\), which does not take away generality as one may add intermediate points if the polygon contains any antipodal edge.

By concatenating all the edge paths, the spherical polygon \(\Gamma\) is naturally associated with a continuous closed path \(C_\Gamma\colon\SS^1\to\SS^2\).  
Let \(S_\Gamma\colon\DD^2\to\SS^2\) be any smooth extension%
\footnote{A smooth extension $S_\Gamma\colon\DD^2\to\SS^2$ exists for any smoothly parameterized path $C_{\Gamma}\colon\SS^1\to\SS^2$.  A smooth path $C_{\Gamma}$ can be constructed by concatenating smooth parameterizations of the great circular edges with vanishing derivatives \(0=\partial_tC_\Gamma = \partial_t^2C_\Gamma= \cdots\) at the vertices.}
of \(C_{\Gamma}\colon\SS^1=\partial\DD^2\to\SS^2\) to the unit disk \(\DD^2\).  We call \(S_\Gamma\) the \emph{enclosed region} of the spherical polygon \(\Gamma\), which is unique up to an integer number of full wrappings around the entire sphere.  
Define the \emph{signed area} of \(\Gamma\) as the total signed area of \(S_\Gamma\), which is well-defined modulo \(4\pi\):
\begin{align}
\label{eq:AreaDefinition}
    \operatorname{Area}(\Gamma) 
    \coloneqq 
    \iint_{S_\Gamma} \sigma = \iint_{\DD^2}S_\Gamma^*\sigma\in \RR/(4\pi\ZZ).
\end{align}

Here \(S_\Gamma^*\) denotes the pullback via \(S_\Gamma\), and \(\sigma\in\Omega^2(\SS^2)\) is the standard area form of the unit sphere \(\SS^2\subset\RR^3\) induced by the Euclidean metric. The area form \(\sigma\) can be explicitly written as \(\sigma= \sin \theta d\theta\wedge d\phi\) using a spherical coordinate chart, or as \(\sigma=(x dy \wedge dz + y dz \wedge dx + z dx \wedge dy)/(x^2+y^2+z^2)^{3/2} \ |_{\SS^2}\) using the Cartesian coordinates in \(\RR^3\). 
This definition is valid for self-intersecting polygons (See \autoref{fig:signed_area}) and degenerate polygons which may contain edges of zero lengths or consecutive edges that fold back onto each other. 

\begin{figure}[h]
\centering
 \begin{minipage}[t]{0.4\textwidth}
\includegraphics[width=\textwidth]{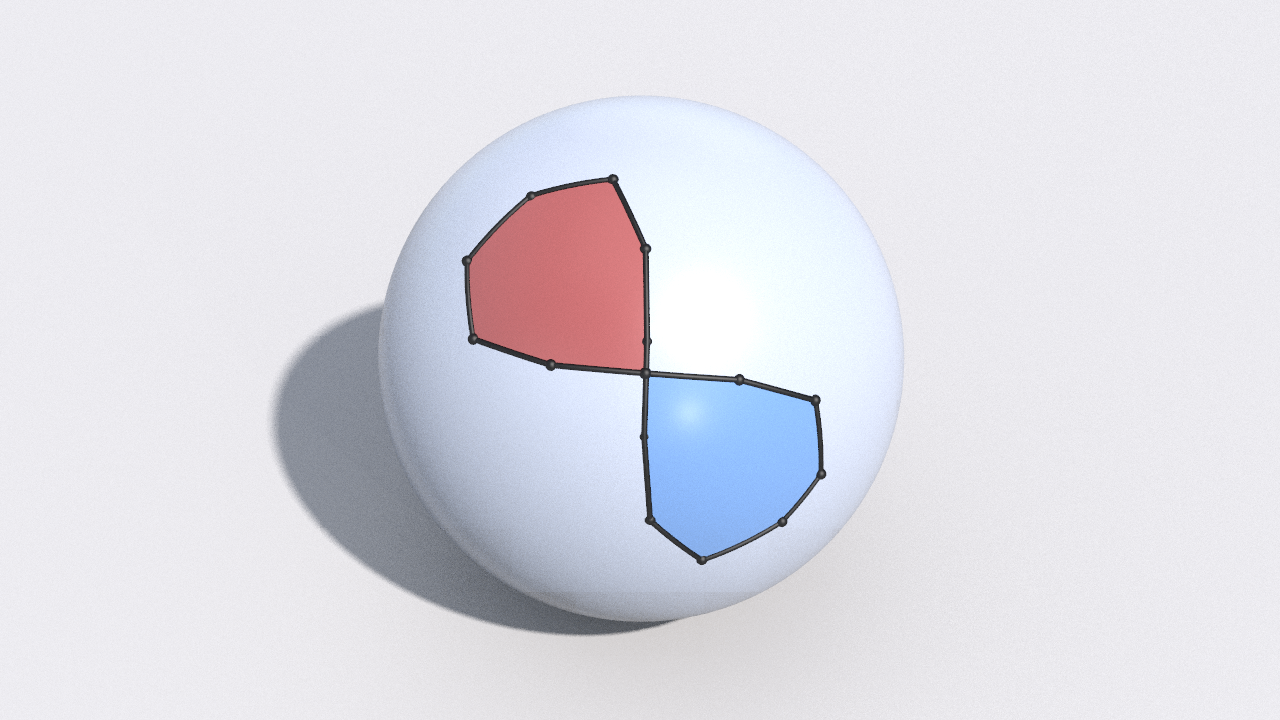}
\end{minipage}
 \begin{minipage}[t]{0.4\textwidth}
\includegraphics[width=\textwidth,angle=0]{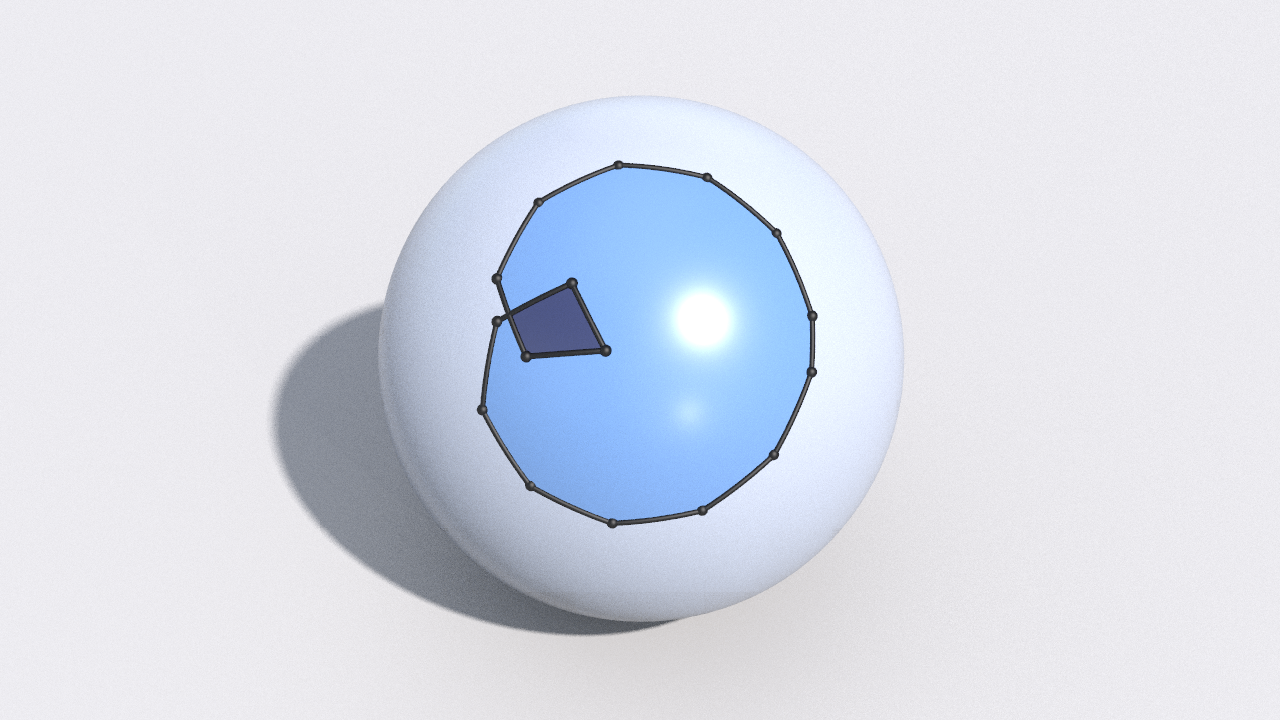}
\end{minipage}
\caption{Examples of spherical polygons. Left: the two regions with different colors contribute positively or negatively to the signed area. Right: the darker region contributes to the signed area twice.}
\label{fig:signed_area}
\end{figure}

\subsection{Areas as line integrals}
Our goal is to find a numerically robust formula for \eqref{eq:AreaDefinition} in terms of the vertex positions \(\Gamma = (p_0,\ldots,p_{n-1})\).
More precisely, an ideal area formula we look for is a line integral ``\(\Area(\Gamma)=\oint_{C_{\Gamma}}\alpha\)'' for some smooth differential 1-form \(\alpha\), analogous to Green's Theorem in the plane.
The reason for this desire goes as follows.
Once \eqref{eq:AreaDefinition} becomes a line integral along the polygonal curve, we can derive the formula for \(\Area(\Gamma)\) by summing the explicit integrals of \(\alpha\) along the great circular arc of each edge.
Such a line-integral-based formula would be applicable to degenerate cases: The angles between consecutive edges would never appear in the formula; the line integral of a smooth 1-form along an edge shrinks to zero gracefully if the edge length shrinks to zero.

Such a line integral formula ``\(\Area(\Gamma) = \oint_{C_\Gamma}\alpha\)'' appears to rely on the exactness of \(\sigma\), which is the existence of a smooth 1-form \(\alpha\) so that \(d\alpha = \sigma\). 
If \(\alpha\) exists, then by Stokes' Theorem \(\Area(\Gamma) = \iint_{S_\Gamma}\sigma = \iint_{S_{\Gamma}}d\alpha = \oint_{C_{\Gamma}}\alpha\).
However, the \emph{spherical area form
\(\sigma\) is not exact.}
Fortunately, and perhaps surprisingly, a line integral formula does not require the exactness of \(\sigma\) as described below.

\begin{definition}[Prequantum bundle]
\label{def:PrequantumBundle}
Let \(\beta\in\Omega^2(\Sigma)\) be a closed 2-form on a manifold \(\Sigma\).  A \emph{prequantum bundle} over \((\Sigma,\beta)\) is a principal circle bundle \(\pi\colon Q\to\Sigma\) equipped with an equivariant 1-form \(\alpha\in\Omega^1(Q)\) with $d\alpha = \pi^*\beta$.
\end{definition}
Intuitively, a principal circle bundle \(Q\) over the base manifold $\Sigma$ is a \((\dim(\Sigma)+1)\)-dimensional space where a circle is attached to each point of \(\Sigma\). The equivariance of \(\alpha\in\Omega^1(Q)\) means that \(\alpha\) is invariant under a uniform rotation in the circle dimension.

\begin{remark}\label{rmk:existence_prequantum_bundle}
In a classical definition for a prequantum bundle, \(\alpha\) is a \emph{connection 1-form}, which has an additional requirement that \(\alpha(V) = 1\) where \(V\) is the generator of the rotation action.
In general, every compact symplectic manifold \((\Sigma,\beta)\) admits a principal circle bundle \((Q,\alpha)\) with a connection 1-form $\alpha$ satisfying \(d\alpha=\pi^* \beta\) upon a rescaling of \(\beta\) so that \(\beta/2\pi\in H_2(\Sigma,\ZZ)\) 
\cite[Theorem~3]{Boothby_Wang_contact} \cite[Proposition~9]{Kobayashi_Kaeler}.
\end{remark}

\begin{proposition}[Lifted Green's theorem]\label{prop:liftedGreenTheorem}
    Let  \(\pi\colon (Q,\alpha)\to(\Sigma,\beta)\) be a prequantum bundle.   
    For each surface \(S\colon\DD^2\to\Sigma\) consider an arbitrary lift \(\widetilde S\colon\DD^2\to Q\), \(\pi\circ\widetilde S = S\).  Then 
    \begin{align*}
        \iint_{ S}\beta = \oint_{\partial \widetilde{ S}}\alpha.
    \end{align*}
\end{proposition}
\begin{proof}
    \(\oint_{\partial \widetilde{ S}}\alpha = \iint_{ \widetilde{ S}}d\alpha =  \iint_{ \widetilde{ S}}\pi^*\beta  =  \iint_{\pi\circ \widetilde S}\beta =  \iint_{ S}\beta\).
\end{proof}

The lifted Green's theorem enables line integral formula for the area of a spherical curve. Suppose we have a prequantum bundle \(\pi\colon (Q,\alpha)\to(\SS^2,\sigma)\) over the sphere \(\SS^2\).  
For each closed curve \(C_{\Gamma}\) on \(\SS^2\), construct an arbitrary lift \(\widetilde C_{\Gamma}\colon\SS^1\to Q\), \(\pi\circ \widetilde C_{\Gamma} = C_{\Gamma}\).  Then
\begin{align}\label{eq:area_via_prequantum_bundle}
    \Area(\Gamma) = \oint_{\widetilde C_\Gamma}\alpha\mod 4\pi.
\end{align}

\begin{remark}
While all the examples we provide in this article are prequantum bundles,  
for \propref{prop:liftedGreenTheorem}
we only need any smooth map \(\pi\colon Q\to\Sigma\) and \(\alpha\in\Omega^1(Q)\) that satisfies \(d\alpha = \pi^*\beta\) and has a general liftability of a topological disk \(S\) to \(\widetilde S\).  In particular, we do not need \(\pi\) to be a circle bundle or \(\alpha\) to be equivariant.
\end{remark}

\subsection{The Hopf fibration}
\label{sec:HopfFibration}
The Hopf fibration \(\pi\colon\SS^3\to\SS^2\) is a prequntum bundle over \(\SS^2\).
We provide its explicit expressions in the quaternion coordinate \(\HH = \{x\ii +y\jj + z\kk + w\,\vert\, (x,y,z,w)\in\RR^4\}\).
Using the quaternion coordiantes \(q\colon \SS^3\hookrightarrow\HH\cong\RR^4\) and \(p\colon \SS^2\hookrightarrow\Im\HH\cong\RR^3\), 
Hopf's bundle projection is given by
\begin{align}
\label{eq:HopfMap}
    \pi\colon \SS^3\to\SS^2,\quad \pi(q)\coloneqq q\ii \conj q.
\end{align}
Note that \(\pi\) is a principal circle bundle with  action \(\lhd\colon\SS^1\times\SS^3\to\SS^3, (\lhd e^{-\ii\theta})q \coloneqq qe^{-\ii\theta}
\)
For \(q\) and \(q'\coloneqq q e^{-\ii\theta}\) on a same fiber, we write \(\arg(\conj{q'}q)=\theta\).

The 1-form \(\alpha\in\Omega^1(\SS^3)\) for the prequantization is expressed as
\begin{align}
\label{eq:Hopf1Form}
\alpha = 2\Re(\ii \conj q dq)=-2\Re(d\conj q q\ii) = 2\Re( dq\ii \conj q) = -2\Re(q\ii d\conj q).
\end{align}
The 1-form \(\alpha/2\) is a connection form as it is equivariant under \(\SS^1\) actions: \((\lhd e^{\ii\theta})^*\alpha = \alpha\) and \({1\over 2}\alpha({d\over d\theta}\vert_{\theta=0}(\lhd e^{-\ii\theta})q)=1\). 
\begin{proposition}
\label{prop:HopfPrequantum}
    The 1-form \(\alpha\in\Omega^1(\SS^3)\) defined in \eqref{eq:Hopf1Form} and the map \(\pi\colon\SS^3\to\SS^2\) defined in \eqref{eq:HopfMap} satisfy \(d\alpha = \pi^*\sigma\), where \(\sigma\in\Omega^2(\SS^2)\) is the standard area form on the unit sphere.  That is, \((\SS^3,\alpha)\) is a prequantum bundle over \((\SS^2,\sigma)\). 
\end{proposition}
This result is known (\cite{bengtsson_zyczkowski_2006}, \cite[Theorem~ 1]{chernSchroedinger}), but we give a proof in \autoref{appendix:proof_of_prequantum_bundles} for completeness. 
With this setting, \propref{prop:liftedGreenTheorem} is now written more concretely for a spherical polygon:
\begin{corollary}
\label{cor:HopfArea}
    The area \(\Area(\Gamma)\) of a spherical polygon \(\Gamma=\{p_i\}_i\) can be evaluated by a line integral
    \begin{align}
        \label{eq:AreaByHopf}
        \Area(\Gamma) = \sum_i\int_{\widetilde C_\Gamma([t_i,t_{i+1}])}\alpha\mod 4\pi
    \end{align}
    where \(\alpha\) is given by \eqref{eq:Hopf1Form} and \(\widetilde C_\Gamma\) is an arbitrary lift \(\widetilde C_\Gamma\colon\SS^1\to\SS^3\) of the polygon curve \(C_\Gamma = \pi\circ\widetilde C_\Gamma \colon\SS^1\to\SS^2\) with \(C_\Gamma(t_i)=p_i\) for each \(i\).
\end{corollary}

\section{The area formula via the Hopf fibration}

In this section, we obtain an explicit formula by evaluating each piece of the line integral in  \eqref{eq:AreaByHopf} along a great circular arc lifted onto \(\SS^3\). 
We first define {\em dihedral} for a pair of spherical points.
\begin{definition}
    For each \(p,p'\in\SS^2\), \(p\neq -p'\), define
    \begin{align}
        \Dihedral(p,p')\coloneqq \sqrt{1+\langle p,p'\rangle\over 2} + {p\times p'\over \sqrt{2+2\langle p,p'\rangle}} \in\SS^3\subset\HH.
    \end{align}
    which is the unit quaternion that represents the minimal rotation that rotates \(p\) to \(p'\) i.e. \(rp\conj r = p'\), \(r = \Dihedral(p,p')\).
\end{definition}
There is a unique (unnormalized) \emph{rotation axis} \(v\in T_1 \SS^3 (= \operatorname{Im}\HH) \) for the minimal rotation for non-antipodal \(p,p'\in \SS^2\). This axis \(v\) is parallel to \(p\times p'\) and  \( \Dihedral(p,p')=e^{\frac{v}{2}}\) by the exponential map on  \(\SS^3\).
We also note that  any \(q\in\pi^{-1}p\), the point \(\Dihedral(p,p')q\) is on \(\pi^{-1}p'\). 

We will see that each summand in \eqref{eq:AreaByHopf} can be explicitly expressed in terms of the dihedral after recognizing how the dihedral represents the horizontal lift over a great circular arc.

\begin{proposition}[Horizontal lift on \((\SS^3,\alpha)\)]\label{prop:horitontal_lift}
Consider the Hopf fibration \(\pi\colon(\SS^3,\alpha)\allowbreak \to (\SS^2,\sigma)\) described in \secref{sec:HopfFibration}.
Let \(p_0, p_1\) be two arbitrary non-antipodal points on \(\SS^2\), let
    \(\gamma\colon [0,1]\to \SS^2\) be the great circular arc joining \( p_0 \) and \(p_1\), and let \(v\in T_1 \SS^3\) be the imaginary quaternion such that \(e^{\frac{v}{2}}= \Dihedral(p_0,p_1)\).
Then for each given point \(q\in \pi^{-1}p_0\), the horizontal lift \(\tilde{\gamma}_H\) of \(\gamma\)  with \(\tilde{\gamma}_H(0)=q\) is given by 
\begin{align*}
    \tilde{\gamma}_H (t)= e^{\frac{vt}{2}} q.
\end{align*}
\begin{proof}
It follows from  
\(\gamma(t)= e^{\frac{vt}{2}} p_0 e^{-\frac{vt}{2}}\)
that \(\tilde{\gamma}_H\) is a lift over \(\gamma\) with respect to \(\pi\).
We now show that \(\tilde{\gamma}_H\) is horizontal. 
At each \(q\in \SS^3\), the tangent space and the horizontal subspace with respect to 
the connection 1-form \(\alpha/2\) are \(T_q \SS^3=\operatorname{Span}( q\ii, q\jj, q\kk)\) and \(
H_q \SS^3=\operatorname{Span}( q\jj, q\kk)\) respectively.
We have for each \(t\) that
\begin{align*}
    \frac{d}{dt}\tilde{\gamma}_H (t)
    = \frac{1}{2} e^{\frac{vt}{2}} v q.
\end{align*}
Here we note that the conjugation \( h \mapsto q h \bar{q}\) is an isometry in \(\operatorname{Im}(\HH)\), from which it follows that 
\(v\in q (H_q \SS^3) \bar{q}\)
as
\begin{align*}
    0=\langle v,p\rangle=\langle q (\bar{q} v  q)\bar{q}, q\ii \bar{q}\rangle.
\end{align*}
Therefore, \(e^{vt\over 2}vq\) is horizontal.
\end{proof}
\end{proposition}

\begin{lemma}
\label{lem:SegmentIntegral}
Assume \(\pi\colon(\SS^3,\alpha)\to(\SS^2,\sigma)\), \(p_0,p_1\in \SS^2\), and \(\gamma\colon [0,1]\to \SS^2\)  as in \propref{prop:horitontal_lift}.
    Then for any lift \(\tilde\gamma\colon[0,1]\to\SS^3\) of \(\gamma\), 
    the line integral \(\int_{\tilde\gamma}\alpha\) is explicitly given by
        \begin{align}
        \int_{\tilde\gamma}\alpha = 2\arg\bigl(\conj{q_1}\Dihedral(p_0,p_1)q_0\bigr), 
    \end{align}
    where \(q_0 \coloneqq \tilde\gamma(0)\), \(q_1  \coloneqq \tilde\gamma(1)\).
    \begin{proof}
  \propref{prop:horitontal_lift} asserts that \(q_0\in \pi^{-1}p_0\) 
  and \(q_H\coloneqq \Dihedral(p_0,p_1) q_0 \in \pi^{-1} p_1\) are connected by the unique horizontal lift \(\tilde{\gamma}_H\).
Let us consider a parametric surface $\Delta \subset \SS^3$ given by
\begin{align*}
    \Delta \coloneqq \left\{  \tilde{\gamma}_H(t) e^{\ii\theta} \in \SS^3 \big| t\in[0,1], \theta\in[0,\theta_t] \  \right\},
\end{align*}
where \(\theta_t\) for each \(t\) is the angular difference  \(\theta_t\coloneqq \arg\left(\overline{\tilde{\gamma}(t)}\tilde{\gamma}_H(t)) \right)\).  We first obtain that 
\begin{align*}
    \int_{\partial\Delta} \alpha 
    = \int_\Delta d\alpha 
    = \int_{\Delta} \pi^*\sigma =0,
\end{align*}
as \(\pi(\Delta)=\gamma([0,1])\).

Note that the boundary \(\partial \Delta\) is a  closed path consisting of three segments: (I).  the \sttext{lift} lifted path \(\tilde{\gamma}\) from \(q_0\) to \(q_1\); (II). the vertical path \(\{q_1 e^{\ii\theta} \ |\ \theta\in [0,\theta_1] \}\) from  \(q_1\) to \(q_H\);  (III). the horizontal lift \(\tilde{\gamma}_H\) from \(q_H\) to \(q_0\). 
Since the integral of \(\alpha\) along the third path makes no contribution, we have
\begin{align*}
    \int_{\tilde{\gamma}} \alpha 
    &= - \int_{\{q_1 e^{\ii\theta} \ | \ \theta\in [0,\theta_1] \}}\alpha 
    = \int_{\{q_1 e^{-\ii\theta} \ | \ \theta\in [-\theta_1,0] \}}\alpha 
    = \int ^0 _{-\theta_1} \alpha 
    \left(\frac{d}{d\theta} (\lhd e^{-\ii\theta}) q_1 \right) d\theta \\
    &= \int ^0 _{-\theta_1}  2 d\theta
    =  2\arg\left(\conj{q_1}q_H\right),
\end{align*}
which concludes the proof. 
\end{proof}
\end{lemma}

As a direct result of  \corref{cor:HopfArea} and \lemref{lem:SegmentIntegral}, we obtain our main theorem:

\begin{theorem}[Area formula via the Hopf fibration]\label{th:area_formula}
    Let \(\Gamma = (p_0,\ldots,p_{n-1}), p_i\in\SS^2, i\in \ZZ_n\), be a spherical polygon.  For each \(i\in\ZZ_n\), pick an arbitrary lift \(q_i\in\SS^3\)  i.e. \(\pi(q_i) = p_i\).  Then,
    \begin{align}
    \label{eq:FinalAreaFormula}
        \Area(\Gamma) = 2\sum_{i=0}^{n-1}\arg\bigl( \conj{q_{i+1}}\Dihedral(p_{i},p_{i+1})q_i\bigr)\mod 4\pi.
    \end{align}
\end{theorem}

The formula requires an arbitrary lift \(q_i\in\SS^3\) of the vertex positions \(p_i\in\SS^2\) for \(i\in\ZZ_n\). An example is
\begin{align}
\label{eq:ChoiceOfLift}
    q_i = 
    \begin{cases}
        \Dihedral(\ii,p_i),& \langle p_i,\ii\rangle \geq 0,\\
        \Dihedral(-\ii,p_i)\jj,& \langle p_i,\ii\rangle <0,
    \end{cases}
\end{align}
which is uniquely defined globally.

As a special case of \autoref{th:area_formula}, choosing the horizontal lift results in no contribution of \(\alpha\) except for the endpoint of the polygon.
\begin{corollary}[Area formula by  
the horizontal lift ]\label{cor:horizontal_lift_area_formula}
Let \(q_0\) be a point in the fiber \(\pi^{-1}p_0\) and let us inductively define \(q_{i+1}:=\Dihedral(p_i,p_{i+1}) q_i \) for \(i=0,\ldots,n-1\). Then we have
\begin{align}
    \operatorname{Area}(\Gamma) 
    =2\arg(\conj{q_{0}} q_n  ).
\end{align}

\end{corollary}
At the end of the section, we make a remark regarding the numerical stability of the formula \eqref{eq:FinalAreaFormula}.
\begin{remark}
    The branching discontinuities in the ``\(\arg\)'' function in \eqref{eq:FinalAreaFormula} and the ``if'' statement in \eqref{eq:ChoiceOfLift} are smooth in the mod-\(4\pi\) arithmetic of \eqref{eq:FinalAreaFormula}.
The only calculation that can be numerically unstable is the evaluation of the \(\Dihedral\) function when the two arguments are close to antipodal.
This antipodal dihedral evaluation is avoided by the choice \eqref{eq:ChoiceOfLift}.
The entire evaluation of our area formula \eqref{eq:FinalAreaFormula} with \eqref{eq:ChoiceOfLift} is numerically stable as long as we do not have antipodal edges where \(\langle p_i,p_{i+1}\rangle \approx-1\), which is easily preventable by inserting a midpoint to any close-to-antipodal edge.
\end{remark}

\section{Derivation of the classical formula by \(\SO(3)\) as a prequantum bundle}\label{sec:derivation_classical_formula}
The Hopf fibration structure can also be seen in the group \(\SO(3)\) of 3D rotations. In fact, the classical formula \eqref{eq:GB_formula} can be interpreted as a special case of the lifted Green's theorem on \(\SO(3)\) using a specific lift not as numerically stable as our formula (\thmref{th:area_formula}) using either the lift \eqref{eq:ChoiceOfLift} or the horizontal lift (\corref{cor:horizontal_lift_area_formula}).
We see that \(\SO(3)\) as the unit tangent bundle over \(\SS^2\) is also a prequantum bundle with a specific connection form. As \(\SS^3\) is a double cover of \(\SO(3)\), the Hopf fibration \(\pi:\SS^3\rightarrow \SS^2\) has a decomposition \(\pi=\pi_2 \circ \pi_1\) given by,
\begin{align}
\pi_1\colon& \SS^3 \rightarrow \SO(3)\\
    &q \mapsto (q\ii\bar{q}, q\jj\bar{q}, q\kk\bar{q}), \nonumber 
\end{align}
and 
\begin{align}\label{eq:SO3map}
\pi_2\colon& \SO(3) \rightarrow \SS^2 \\
    &(p_1,p_2,p_3) \mapsto p_1,  \nonumber 
\end{align}
where each element of \(\SO(3)\) is represented by three column vectors. 

The tangent space \(T_P \SO(3)\) at each \(P\in \SO(3)\) is \(dL_P \mathfrak {so}(3)=\left\{PW | W^T=-W\right\}\). We identify   each 
\(W = 
\begin{psmallmatrix}
0&-\omega_3&\omega_2\\
\omega_3&0&-\omega_1\\
-\omega_2&\omega_1&0
\end{psmallmatrix} \in  \mathfrak {so}(3)\)  with \(\omega=(\omega_1,\omega_2,\omega_3)\in \RR^3\).
We define a\sttext{n} 1-form \(\eta\) by  
\begin{align}\label{eq:SO3_1Form}
        \eta|_P(PW)\coloneqq -\omega_1.
\end{align}
Then \((\pi_2,\eta)\) is a principal circle bundle with \(\SS^1\) action \(\lhd\colon\SS^1\times \SO(3)\to \SO(3),\) by \((\lhd e^{\ii\theta})P = P
\begin{psmallmatrix}
    1&0&0\\
    0&\cos\theta&-\sin\theta\\
    0&\sin\theta&\cos\theta
\end{psmallmatrix}.
\)
\begin{proposition}\label{prop:SO3Prequantum}
        The 1-form \(\eta\in\Omega^1(\SO(3))\) defined in \eqref{eq:SO3_1Form} and the map \(\pi_2\colon \SO(3)\allowbreak\to\SS^2\) 
        defined in \eqref{eq:SO3map} satisfy \(d\eta = \pi_2^*\sigma\), where \(\sigma\in\Omega^2(\SS^2)\) is the standard area form on the unit sphere.  That is, \((\SO(3),\eta)\) is a prequantum bundle over \((\SS^2,\sigma)\). 
\end{proposition}
We give a proof in \autoref{appendix:proof_of_prequantum_bundles}.
To recover the classical formula via \(\pi_2\colon(\SO(3),\eta)\allowbreak\rightarrow(\SS^2, \sigma)\),
we define a lift as follows. We take for each \(p_i\) the forward velocity from \(p_i\) to \(p_{i+1}\) on \(\SS^2\).  It is given at \(p_i\eqqcolon\gamma(t_i)\) by
\begin{align*}
    v(p_i)\coloneqq\lim_{h\to +0}\frac{\gamma(t_i+h)-\gamma(t_i)}{h}\in T_{p_i}\SS^2,
\end{align*}
which is a positive multiple of \(-p_i\times(p_i\times p_{i+1})\).
Then
\begin{align*}
   \tilde{\gamma}(t_i)\coloneqq\left(p_i, \frac{v(p_i)}{|v(p_i)|}, p_i \times \frac{v(p_i)}{|v(p_i)|} \right)
\end{align*} 
defines a lift \( \tilde{\gamma} \colon \SS^1 \rightarrow \SO(3)\) of \(\gamma\). For \(\tilde{\gamma} \), we have 
\begin{align}\label{eq:classical_from_SO3}
    \int_{\tilde\gamma}\eta =-\sum_i \vartheta_i
\end{align}
with the exterior angles \(\vartheta_i\) given in \eqref{eq:exterior_angle1}.
One has to be cautious when drawing conclusion from \eqref{eq:classical_from_SO3} about the area formula using \propref{prop:liftedGreenTheorem}.  In fact, \(\Area(\Gamma) = 2\pi + \int_{\tilde\gamma}\eta = 2\pi-\sum_i\vartheta_i\) noting the extra term of \(2\pi\) (\cf\@ \eqref{eq:GB_formula}).
This is because \(\tilde\gamma\) is a non-contractible loop in \(\SO(3)\) which is not the boundary of a disk.
To obtain the classical formula \eqref{eq:GB_formula}, lift \(\tilde\gamma\colon\SS^1\to\SO(3)\) to \(\hat\gamma\colon \SS^1\to\SS^3\) by the universal cover \(\pi_1\colon\SS^3\to\SO(3)\).  Note that \(\alpha = \pi_1^*\eta\), and that \(\hat\gamma(0)=\hat\gamma(2\pi)\) and \(\lim_{t\nearrow 2\pi}\hat\gamma(t)\) has an angle difference of \(\pi\) in the fiber \(\pi^{-1}(\gamma(0))\).  

The classical formula has a variant that locates a pole \(Z\in\SS^2\) and sums up the signed area of triangles \((p_i,p_{i+1},Z)\). This formula is given as, 
\begin{align}\label{eq:gauss_bonnet}
    \operatorname{Area}(\Gamma) 
    =\sum_i \operatorname{sign}\left(\det(p_i,p_{i+1},Z)\right) \operatorname{UnsignedArea}(p_i,p_{i+1},Z),
\end{align}
where the unsigned area of each spherical triangle \((x_0,x_1,x_2)\) is computed as,
\begin{align}\label{eq:unsigned_triangle_area}
    \operatorname{UnsignedArea}(x_0,x_1,x_2)=-\pi + \sum_{i\in \ZZ_3} \arccos\left(\frac{x_{i-1}\times x_{i}}{|x_{i-1}\times x_{i}|}\cdot\frac{x_i\times x_{i+1}}{|x_i\times x_{i+1}|}\right).  
\end{align}
This formula can also be recovered by the Hopf fibration. 
Setting the lift \(q_i\coloneqq \Dihedral(Z,p_i)\)
for each \(p_i\), we obtain \eqref{eq:gauss_bonnet}. Numerically, this formula is unstable if any of the vertices is close to \(Z\) or \(-Z\), which is explained by the numerical sensitivity of \(\Dihedral(Z,p_i)\) and
\(\frac{p_{i}\times Z}{\|p_{i}\times Z\|}\).

\section{Numerical examples}

In this section, we present numerical examples of area computation for spherical polygons using our formula \eqref{eq:FinalAreaFormula}. 
Moreover, we demonstrate how this formula can also be utilized to determine the total torsion of a space curve, which differs from \(2\pi\) exactly by the enclosed area of the spherical curve traced out by the tangents of the curve.
By comparing the results obtained using both our formula and the classical formula \eqref{eq:GB_formula}, we show that our formula produces consistent and converging solutions, even for singular curves.
This improved numerical robustness allows for more accurate measurements of spherical areas and total torsion.

We employ the horizontal lift approach (\corref{cor:horizontal_lift_area_formula}) for computation in all of our examples.
For a given closed spherical curve \(\gamma\colon [0,2\pi)\rightarrow \SS^2\), 
we use a uniform division \(\{t_i\coloneqq {2\pi i\over n}\}_{i=0}^{n-1}\) of the interval \([0,2\pi)\) to specify the vertices \(\{\gamma(t_i)\}_{i=0}^{n-1}\) with some positive integer \(n\).  This process turns the spherical curve into a spherical polygon.

\subsection{Spherical cardioid}
\label{sec:SphericalCardioid}
We compute the area of a spherical curve \(\gamma\) given by the stereographic projection image \(\gamma = P\circ\gamma_{\RR^2}\) of a planar cardioid \(\gamma_{\RR^2}\) (\autoref{fig:cardioid}).  
Explicitly, the planar cardioid is parametrically  given by
\begin{align*}
    \gamma_{\RR^2}(t)=\left( 2(1-\cos(t))\cos(t), 2(1-\cos(t))\sin(t) \right),
\end{align*}
and the stereographic projection from the plane to the sphere is
\begin{align*}
    P\colon (x,y)\mapsto \frac{1}{x^2+y^2+1}(2x,2y,x^2+y^2-1).
\end{align*}
We compute the area with various numbers \(n\) of vertices using our formula and the classical formula (\ref{eq:GB_formula},\ref{eq:exterior_angle1}).
\autoref{fig:plot_cardioid_Frenet_spiral} (left) shows their numerical results.
Note that this spherical cardioid has a cusp, \ie\@ \(\partial_t\gamma\) changes sign at \(t = 0\).
As the polygon refines (\(n\to\infty\)), the edge lengths adjacent to \(\gamma(t_0)\) decrease to zero superlinearly, and \(\frac{\gamma(t_{n-1})\times \gamma(t_{0})}{|\gamma(t_{n-1})\times \gamma(t_{0})|}\cdot\frac{\gamma(t_{0})\times \gamma(t_{1})}{|\gamma(t_{0})\times \gamma(t_{1})|}\rightarrow -1\).
These conditions make the classical formula numerically unstable as observed in \autoref{fig:plot_cardioid_Frenet_spiral}, left.  
In contrast, our formula is numerically stable despite the presence of the cusp.
\begin{figure}[ht]
\centering
 \begin{minipage}[t]{0.4\textwidth}
\includegraphics[width=\textwidth]{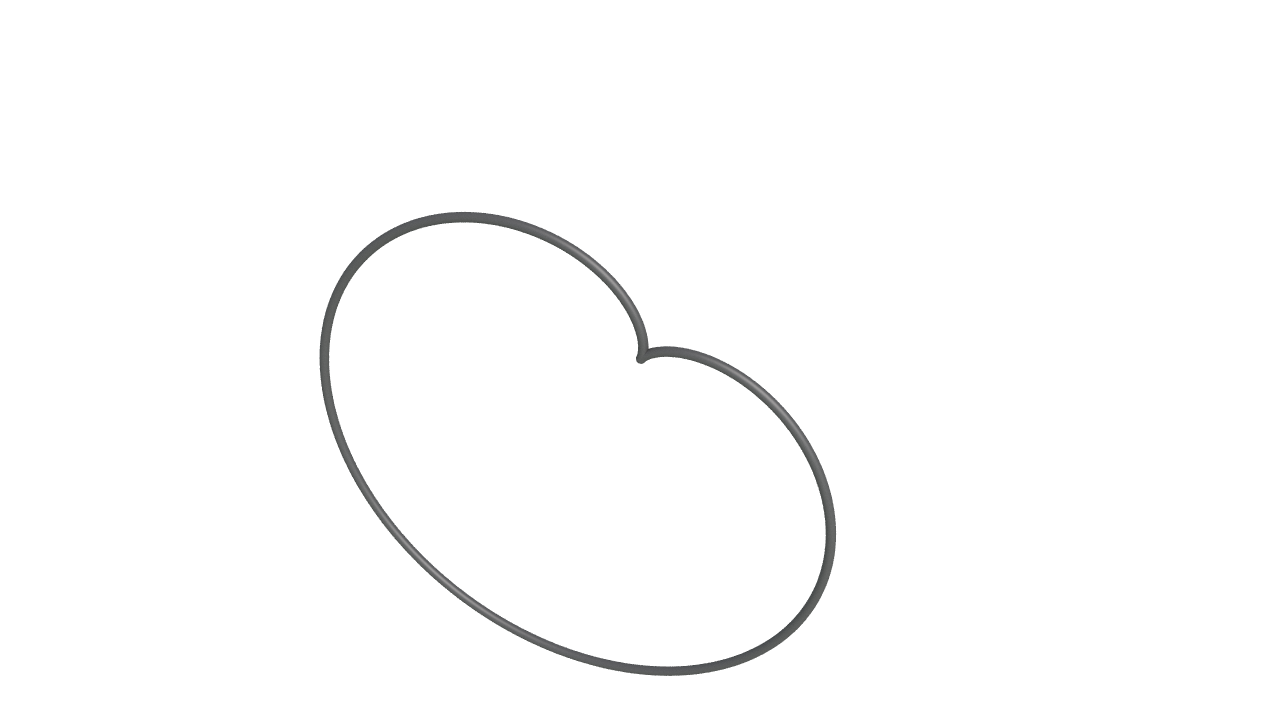}
\end{minipage}
 \begin{minipage}[t]{0.4\textwidth}
\includegraphics[width=\textwidth,angle=0]{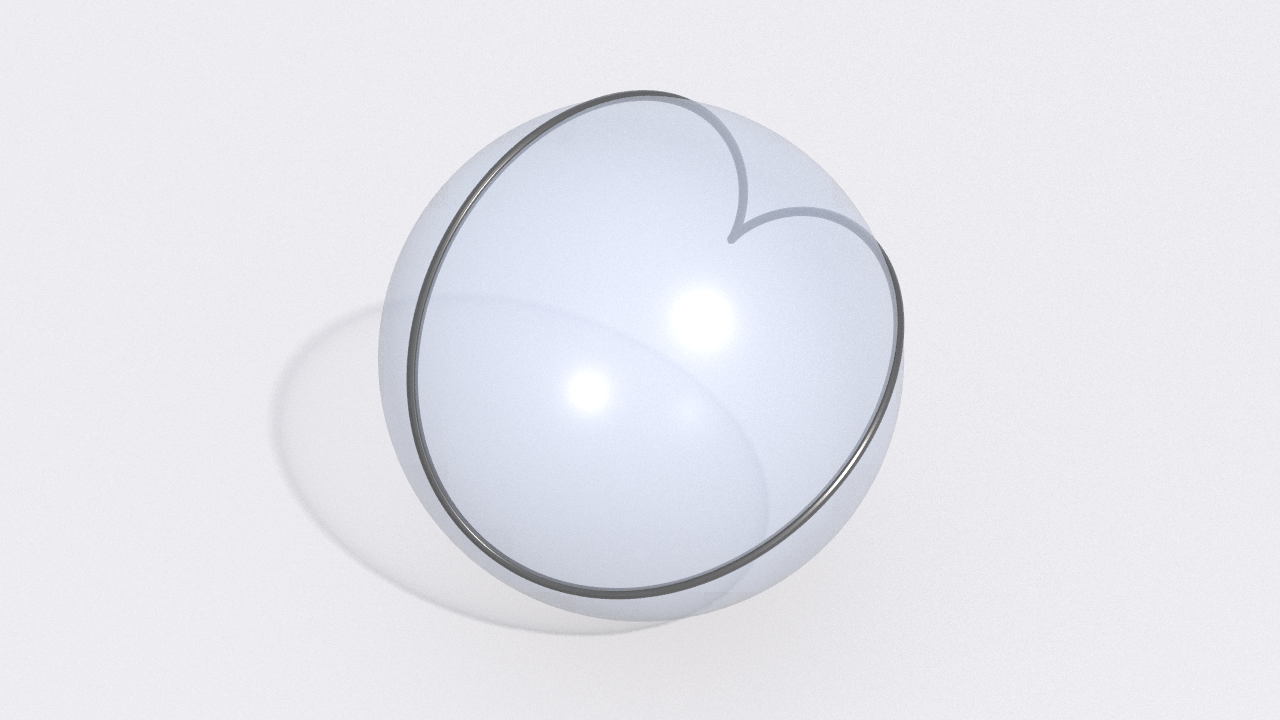}
\end{minipage}
\caption{ Cardioid (left) stereographically projected on the sphere (right). }
\label{fig:cardioid}
\end{figure}

\begin{figure}[htbp]
\centering
 \begin{minipage}[t]{0.45\textwidth}
\includegraphics
 [keepaspectratio, scale=1.]
{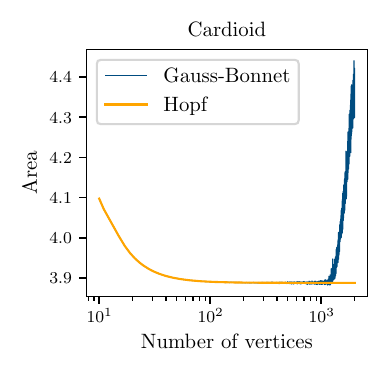}
\end{minipage}
 \begin{minipage}[t]{0.45\textwidth}
\includegraphics
 [keepaspectratio, scale=1.]
{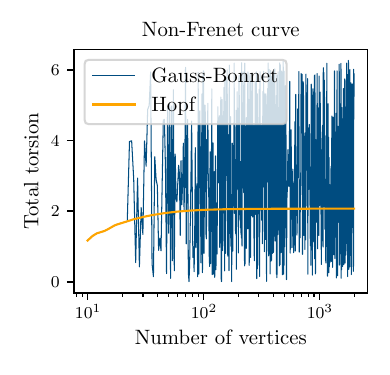}
\end{minipage}
\caption{
Numerical values of the signed areas of spherical curves discretized into spherical polygons with increasing number of vertices, computed using the classical formula (Gauss--Bonnet) and our formula (Hopf).
Both formulae give consistent values when the number of vertices is small, but the classical formula becomes unstable as the number of vertices increases.
Left: The area enclosed by a spherical cardioid (Section~\ref{sec:SphericalCardioid}).
Right: The total torsion \((2\pi- \operatorname{Area}(\gamma'))\) of a non-Frenet space curve (Section~\ref{sec:NonFrenet}).
}
\label{fig:plot_cardioid_Frenet_spiral}
\end{figure}

\subsection{Total torsions of space curves}
\label{sec:TotalTorsion}
Our next examples are about total torsions of closed space curves. For a space curve $\gamma\colon \SS^1\to\RR^3$, 
the total torsion \cite[Ch.\@ 1 Sec.\@ 5.1]{PinkallGross2024geometry}
(equivalently, the total writhe up to a minus sign and a multiple of $2\pi$) can be evaluated as, 
\begin{align}\label{eq:def_total_torsion}
    \operatorname{Torsion}(\gamma) = 2\pi - \Area(\gamma') \quad \mod 2\pi,
\end{align}
where \(\Area(\gamma')\) is the signed area of the unit velocity map given by \(\gamma'=\partial_t\gamma/|\partial_t\gamma|\).
This notion of total torsion also works for a space polygon \(\{\gamma(t_i)\}_i\).  For a space polygon, the unit velocity \( \gamma'\) is given by the normalized edge vector
\begin{align*}
    \gamma'(t_i)=\frac{\gamma(t_{i+1})-\gamma(t_i)}{|\gamma(t_{i+1})-\gamma(t_i)|},
\end{align*}
which forms a spherical polygon, whose signed area can be evaluated by our formula.  
We compute the total torsions of the figure-eight knot (\autoref{fig:figure-eight})
\begin{align*}
    \gamma(t)= \left((2 + \cos(2t))\cos(3t),(2 + \cos(2t)) \sin(3t), \sin(4t)\right),
\end{align*}
and the trefoil knot (\autoref{fig:trefoil}),
\begin{align*}
\gamma(t)=\left(  \sin(t) + 2 \sin(2t),\cos(t)-2\cos(2t), -\sin(3t)\right).
\end{align*}
With sufficiently many vertices, our results converge to numbers that agree with the results in a previous study \cite{najdanovic2021total}: \(-0.5423\) of the figure-eight knot and \(2.2250\) of the trefoil.
\begin{figure}
\centering
 \begin{minipage}[t]{0.4\textwidth}
\includegraphics[width=\textwidth]{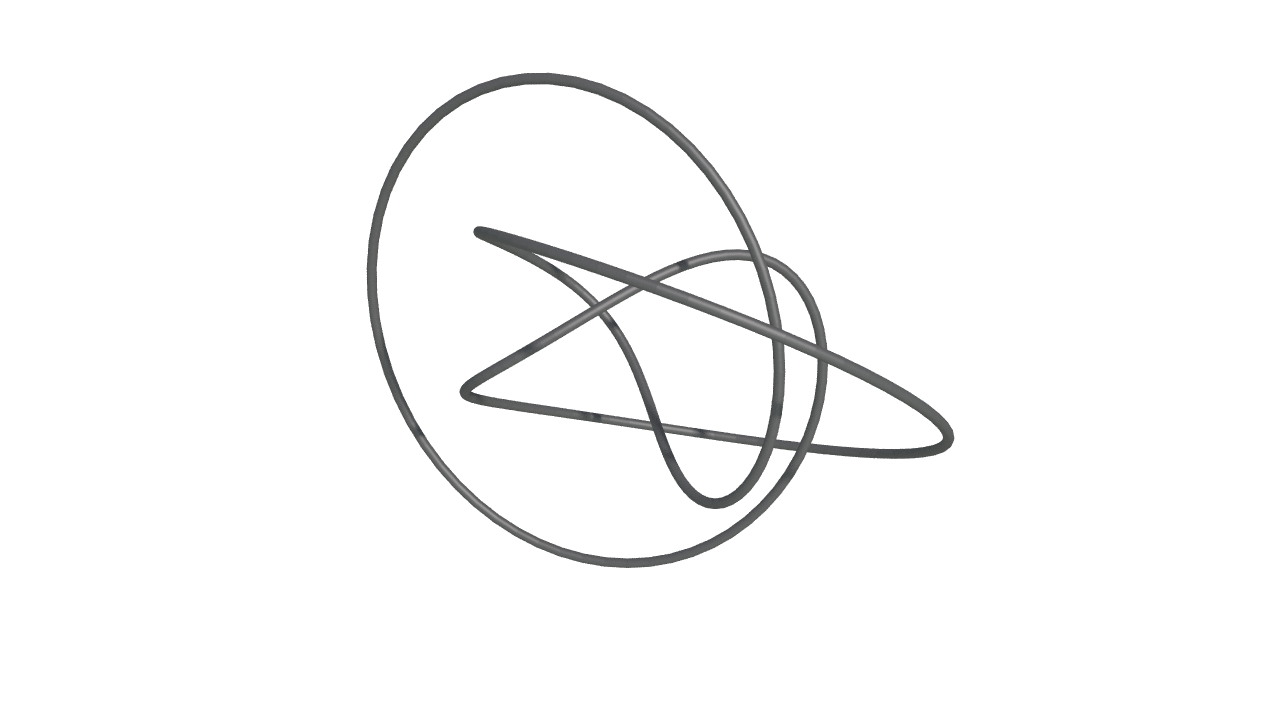}
\end{minipage}
 \begin{minipage}[t]{0.4\textwidth}
\includegraphics[width=\textwidth,angle=0]{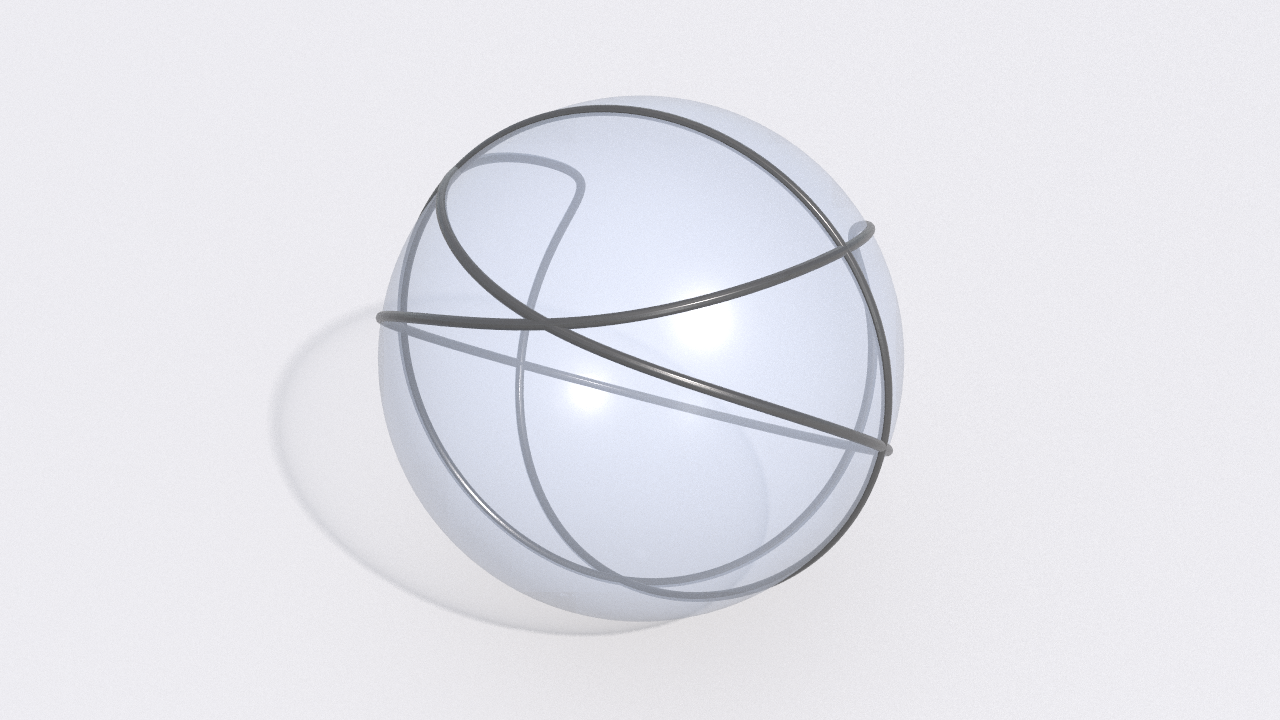}
\end{minipage}
\caption{ Figure-eight knot (left) and its unit velocity map (right). }
\label{fig:figure-eight}
\end{figure}

\begin{figure}
\centering
 \begin{minipage}[t]{0.4\textwidth}
\includegraphics[width=\textwidth]{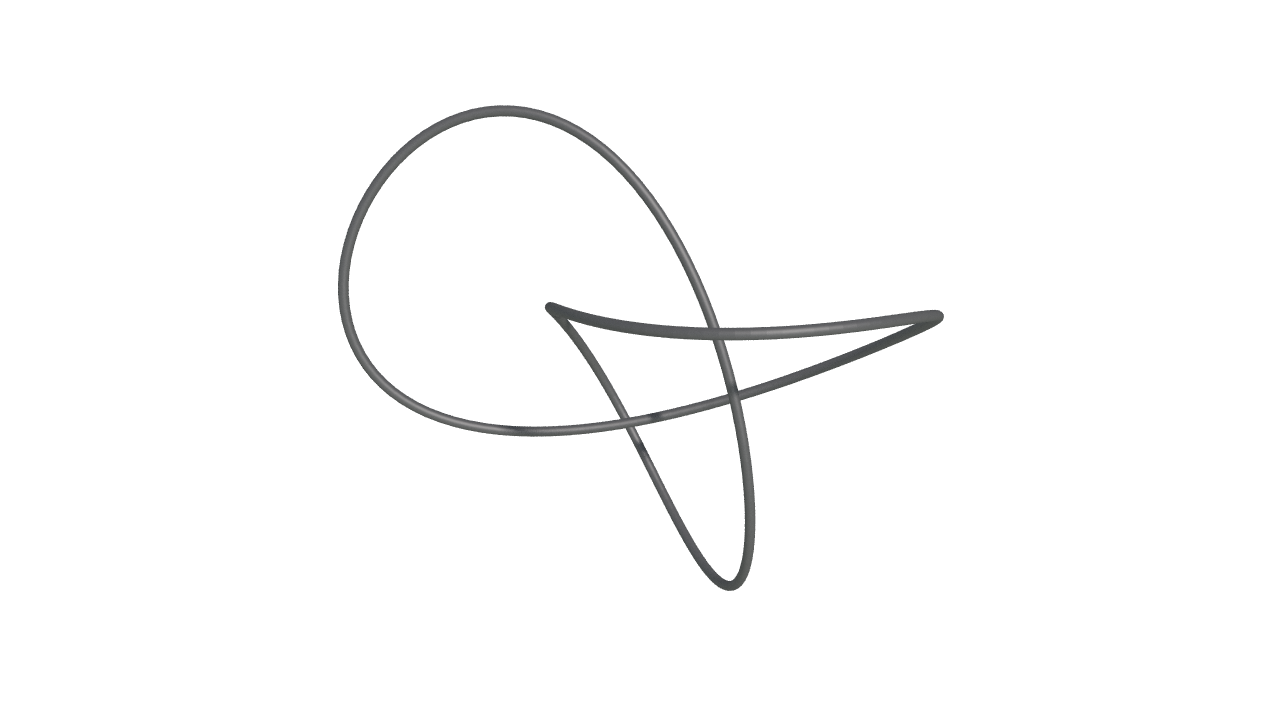}
\end{minipage}
 \begin{minipage}[t]{0.4\textwidth}
\includegraphics[width=\textwidth,angle=0]{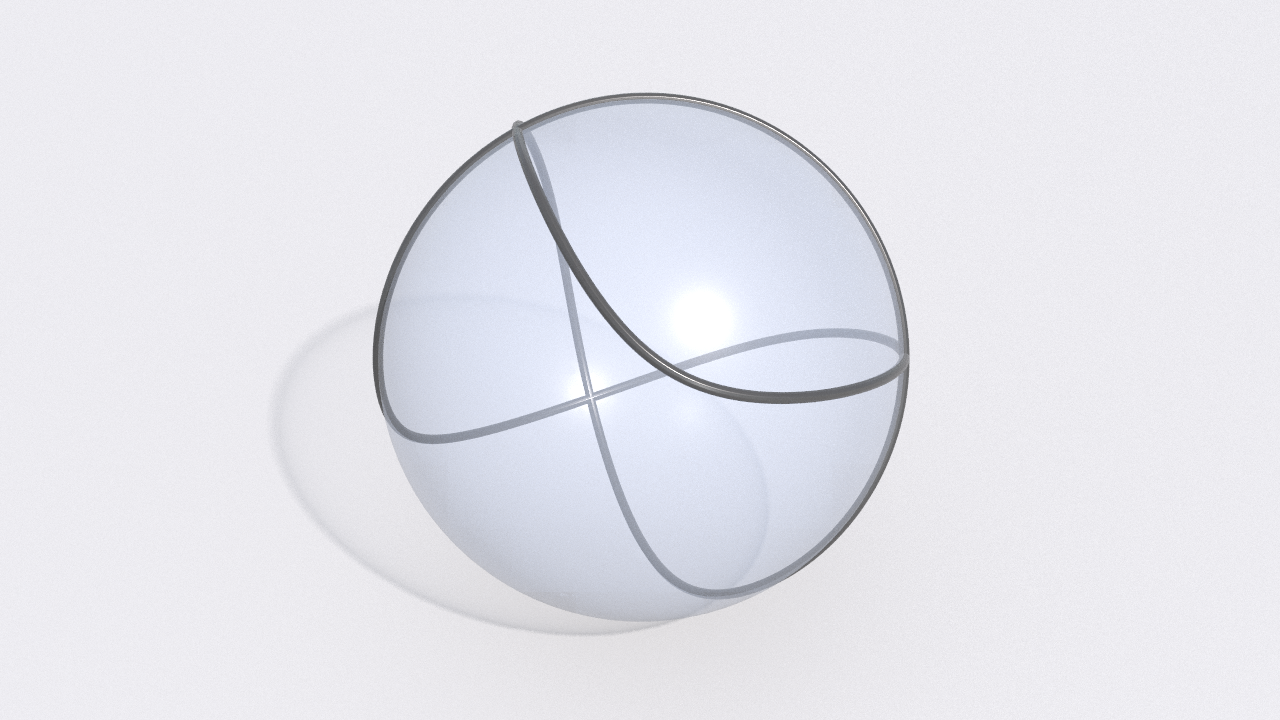}
\end{minipage}
\caption{Trefoil (left) and its unit velocity map (right). }
\label{fig:trefoil}
\end{figure}

\subsection{Total torsion of a non-Frenet curve}
\label{sec:NonFrenet}

In this example we compute the total torsion of a regular closed space curve with a singular (infinitely oscillatory) Frenet--Serret frame. Note that the mod-\(2\pi\) total torsion \eqref{eq:def_total_torsion} and the writhe only require the curve to be regular \ie, \(\partial_t\gamma(t)\neq 0\).  In particular, the curve does not need to possess a regular Frenet torsion.
A smooth regular curve without a regular Frenet--Serret frame is called a non-Frenet curve (See \mbox{\cite[Ch.\@ 1 Sec.\@ 5.6]{PinkallGross2024geometry}} for a detailed discussion). Consider the following example of a smooth non-Frenet closed space curve
\begin{align*}
    \gamma(t')=\frac{1}{e^{-2 t'^2}+t'^2}\left(e^{-t'^2} \cos \left(e^{t'}\right), e^{-t'^2} \sin \left(e^{t'}\right), t'\right), 
\end{align*}
where \(t'\in [-\infty,\infty)\) is a reparametrization of \(t\in [0,2\pi)\) by \(t'=\tan\left( \frac{t-\pi}{2}\right)\).
The spherical curve \(\gamma'\) traced out by the unit velocity displays an exponential spiral about \(t=\pi\) with an infinite turning number and an unbounded geodesic curvature (\autoref{fig:spiral_frenet}).
Note that the total Frenet torsion of \(\gamma\) is the total turning angle of the spherical curve \(\gamma'\), which is divergent.  Despite the divergence of the total Frenet torsion, the mod-$2\pi$ torsion is well-defined since the area enclosed by \(\gamma'\) is bounded.

\begin{figure}[h]
 \begin{minipage}[h]{0.32\textwidth}
\includegraphics[width=\textwidth]{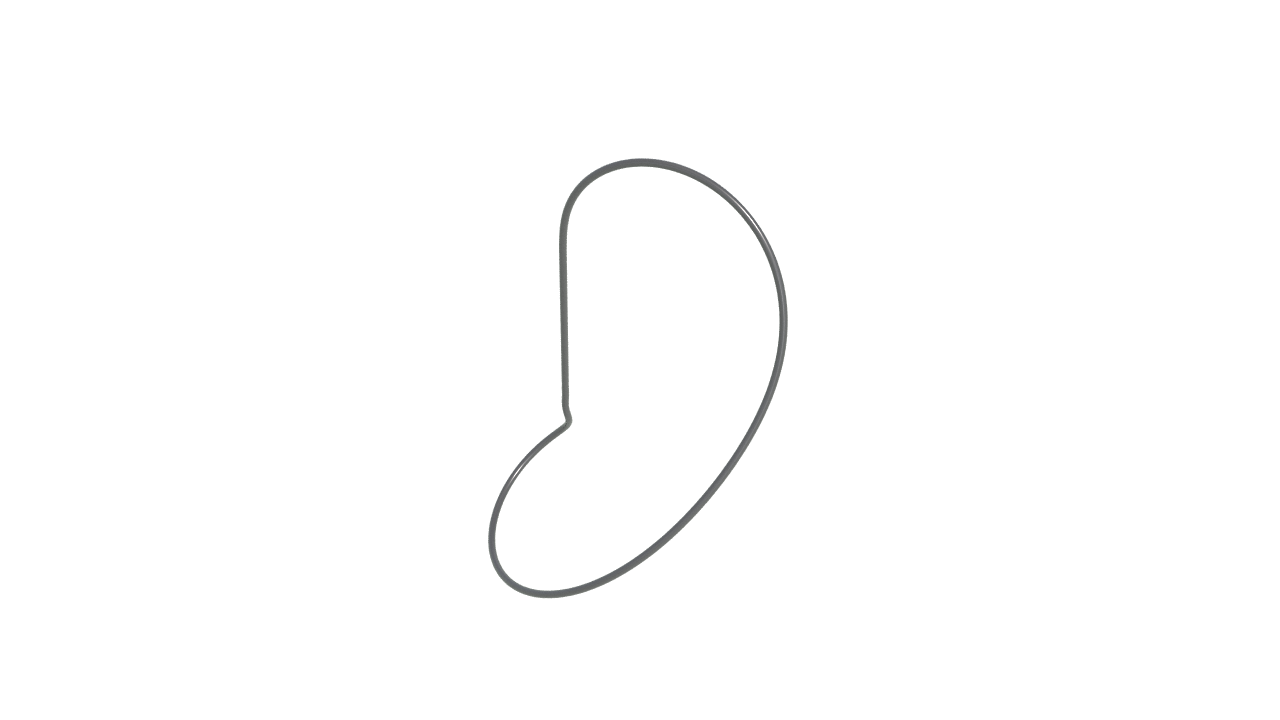}
\end{minipage}
 \begin{minipage}[h]{0.32\textwidth}
\includegraphics[width=\textwidth,angle=0]{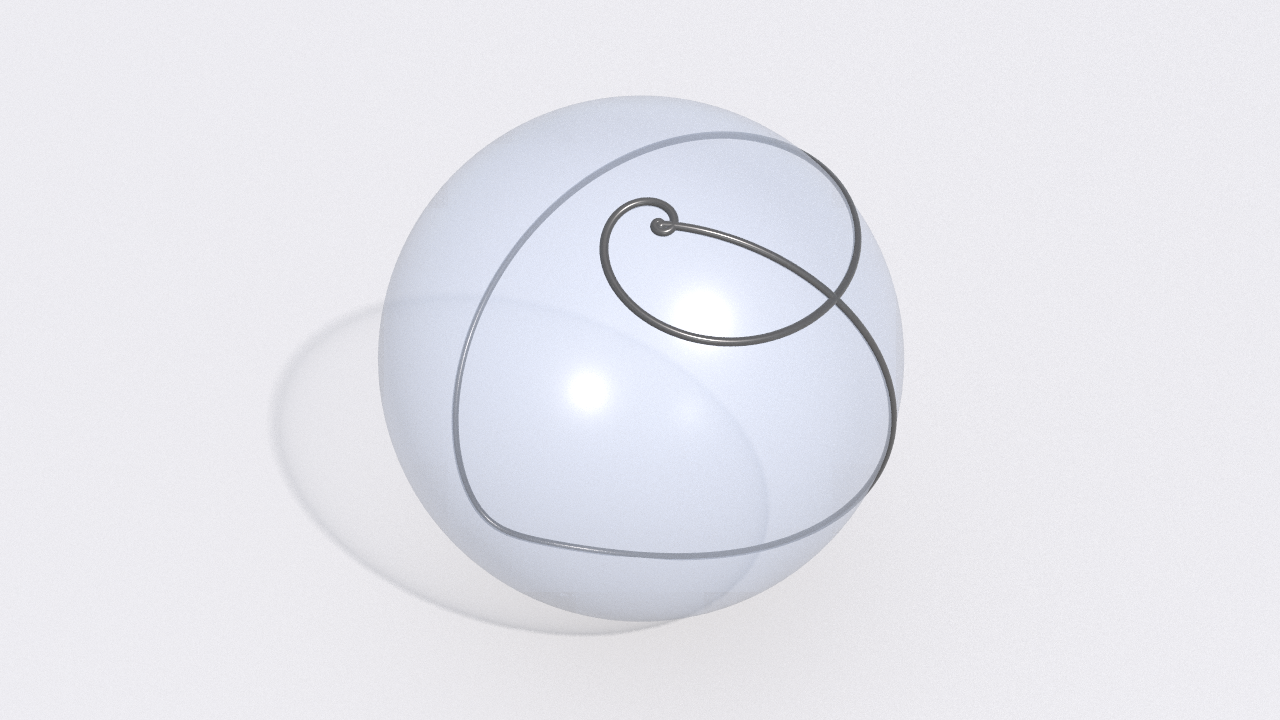}
\end{minipage}
 \begin{minipage}[h]{0.32\textwidth}
\includegraphics[width=\textwidth,angle=0]{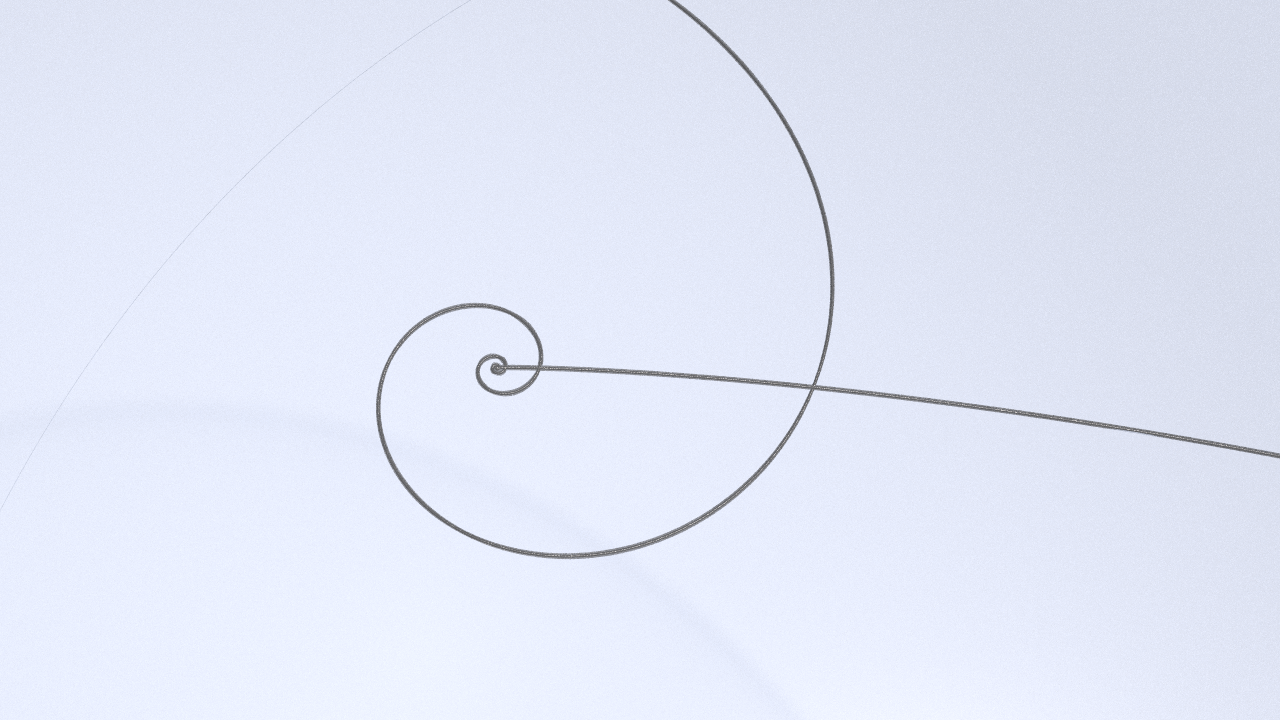}
\end{minipage}
\caption{ A curve with a spiral Frenet-Serre frame (left), its unit velocity map (middle), and the close-up view of the spiral (right).}
\label{fig:spiral_frenet}
\end{figure}

In this example, it is crucial to avoid the classical angle--based formula (\ref{eq:GB_formula} \ref{eq:exterior_angle1}) for evaluating \(\Area(\gamma')\) due to the divergent turning angle in \(\gamma'\).  In fact, evaluating the total torsion using the classical formula is equivalent to sampling and summing the Frenet torsion (exterior angle of the spherical polygon \(\gamma'\)).  The process produces a result that diverges as the number of sample points \(n\to\infty\) (\figref{fig:plot_cardioid_Frenet_spiral}, right).
In contrast, our formula \eqref{eq:FinalAreaFormula} is able to robustly evaluate the total torsion of this non-Frenet curve. 

\subsection{Area of a region on the earth}
\color{black}
\begin{wrapfigure}{hr}{0.3\textwidth}
  \centering
    \includegraphics[width=0.3\textwidth]{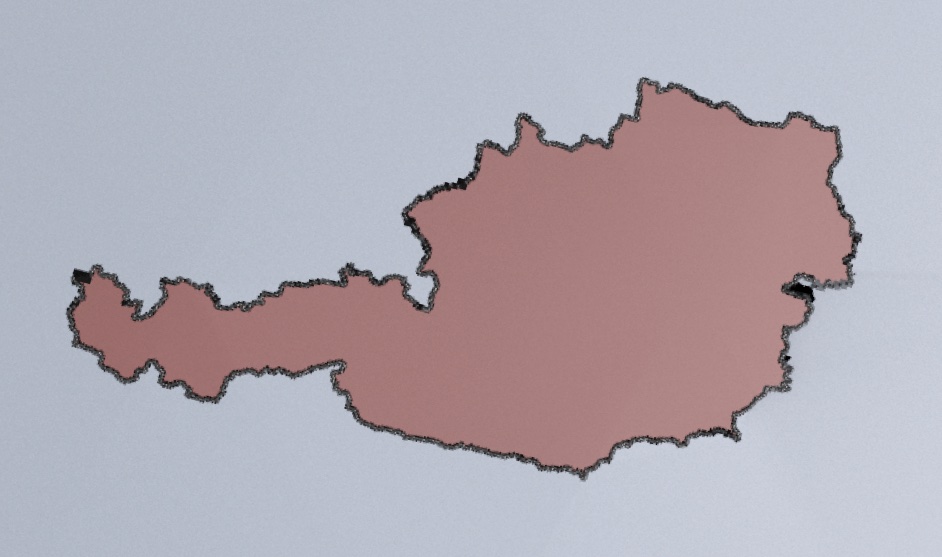}
  \caption{Austria plotted on the sphere.}\label{fig:austria}
\end{wrapfigure}
The next example applies to geography.
We compute an approximate area of Austria using the data from the Database of Global Administrative Areas (GADM) \cite{areas2012gadm}.
The data contains a sequence of latitudes and longitudes (\figref{fig:austria}) forming a spherical polygon. We treat the earth as a round sphere while acknowledging that we neglect its slight ellipsoidal figure and terrains. 

 The solid angle value we computed via our formula was $2.06206\times 10^{-3}$ on the unit sphere. By multiplying the square of the arithmetic mean radius $R\coloneqq (2R_E + R_P)/3\approx 6,371 \operatorname{km}$ \cite{Moritz1980geodetic}, where $R_E, R_P$ are the equatorial and polar radii, we obtain $83,882 \operatorname{km}^2$, which is a descent approximation of the official area $83,871 \operatorname{km}^2$ with  $0.013\%$ relative error. 

\subsection{Solid angle fields and Seifert surfaces}
Our last example demonstrates the construction of the solid angle fields of given space curves (Section~\ref{sec:Introduction}).  
We consider three rectangular loops linked into the topological configuration of Borromean rings (\figref{fig:solidangle}).
Let \(\tilde\gamma\colon\bigsqcup^3\SS^1\to\RR^3\) denote this triplet of space polygons. 
For each point \(\bx\in\RR^3\setminus\tilde\gamma\), we let \(\Omega(\bx)\) be half the solid angle subtended by \(\tilde\gamma\) at \(\bx\).  Explicitly, consider the spherical curves \(\gamma^\bx\colon
{\textstyle\bigsqcup^3\SS^1}\to\SS^2\) given by projecting \(\tilde\gamma\) on the unit sphere centered at \(\bx\): 
\begin{align}
    \gamma^\bx(s) \coloneqq \frac{\tilde \gamma(s)-\bx}{|\tilde \gamma(s)-\bx|},\quad s\in {\textstyle\bigsqcup^3\SS^1}.
\end{align}
The solid angle field \(\Omega\colon\RR^3\setminus\tilde\gamma\to\RR/(2\pi\ZZ)\) is defined by
\begin{align}
    \Omega(\bx)\coloneqq\frac{1}{2}\Area(\gamma^\bx).
\end{align}
Note that the projected spherical curve \(\gamma^\bx\) is degenerate for \(\bx\) on any extended tangent line of \(\tilde\gamma\).  Despite this unavoidable degeneracy, our formula robustly handle the solid angle computation for all \(\bx\) (\figref{fig:solidangle}, left). By extracting a levelset of the solid angle field, we construct a smooth Seifert surface (\figref{fig:solidangle}, right).

\begin{figure}
\centering
 \begin{minipage}[htbp]{0.48\textwidth}
\includegraphics[width=\textwidth, trim = {100, 0, 100, 0}, clip] {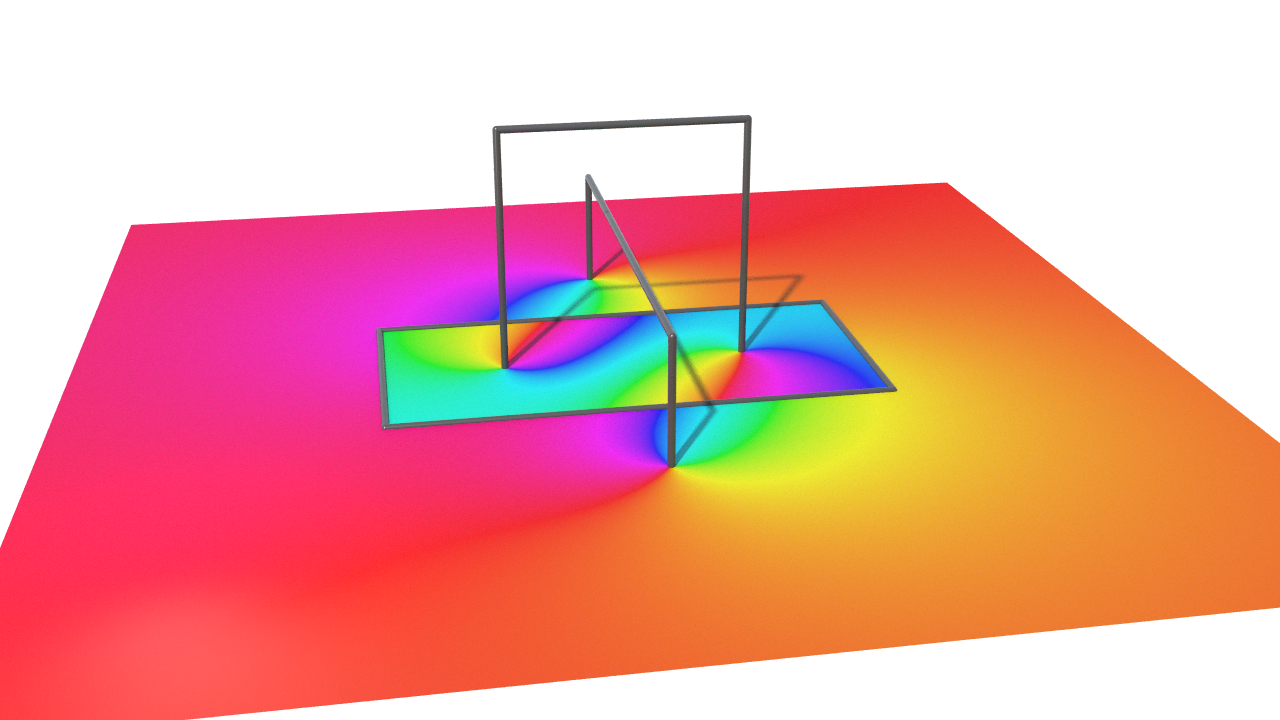}
\end{minipage}
 \begin{minipage}[htbp]{0.48\textwidth}
\includegraphics[width=\textwidth, trim = {100, 0, 100, 0}, clip] {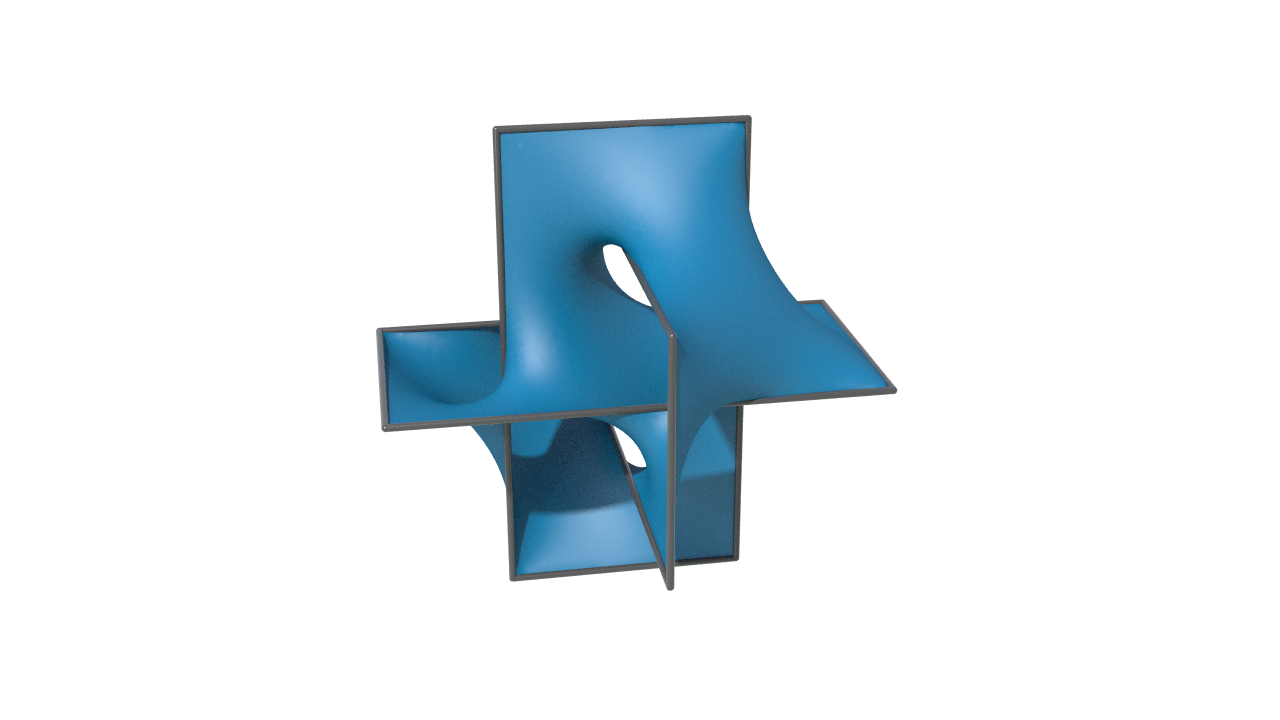}
\end{minipage}
\caption{The solid angle field of the Borromean rings visualized on the $z=0$ plane (left) where each color corresponds to a value in $\SS^1$. A levelset of the solid angle field in $\RR^3$ is a Seifert surface (right).}
\label{fig:solidangle}
\end{figure}

\section*{Concluding remark and outlook}
In this paper, we derived area formulae using Green's theorem on prequantum bundles \(\SS^3\) and \(\SO(3)\) over \(\SS^2\). 
These formulae avoid relying on angle calculation, unlike the classical formula  that fails on degenerate cases.
As the prequantum version of Green's theorem is available for any compact symplectic manifold (Remrak \ref{rmk:existence_prequantum_bundle}), one may investigate area formulae or integral of symplectic form of polygons in other manifolds.

For example, in quantum information; the complex projective space \(\CC\PP^{2^n-1}\)
is regarded as the space of possible states of \(n\)-qubits \cite{bengtsson_zyczkowski_2006} and a closed path is a periodic orbit. We hope that finding an explicit expression of its enclosed area (geometric phase) may lead to practical applications in quantum computation.

\appendix

\section{Prequantum bundles \(\SS^3\) and \(\SO(3)\) over \(\SS^2\) }\label{appendix:proof_of_prequantum_bundles}

\begin{proof}[Proof of  \propref{prop:HopfPrequantum}]
In terms of the quaternion coordinate, the area form of \(\SS^2\) is written as
\begin{align*}
    \sigma = -{1\over 2}\Re( p\Im(dp\wedge dp)).
\end{align*}
On the other hand,
    the differential of \(\pi = q\ii\conj q\) is
    \(d\pi = dq\ii\conj q + q\ii d\conj q\).
    Hence, the pullback area form is computed as
    \begin{align*}
        \pi^*\sigma &= -{1\over 2}\Re(\pi\Im(d\pi\wedge d\pi))
        = -{1\over 2}\Re(q\ii\conj q\Im(d\pi\wedge d\pi)) \\
        &= -{1\over 2}\Re(\ii\conj q\Im(d\pi\wedge d\pi)q) = -{1\over 2}\Re(\ii \Im(\conj qd\pi\wedge d\pi q))
        \\
        &=-{1\over 2}\Re\left(
        \ii\Im((\conj qdq\ii\conj q+\ii d\conj q)\wedge (dq\ii + q\ii d\conj qq))
        \right)\\
        & =-{1\over 2}\Re(\ii\conj q dq\ii\wedge\conj qdq\ii +\ii dq\wedge d\conj q - \ii d\conj q \wedge dq -d\conj q q\ii\wedge d\conj q q),
    \end{align*}
    which agrees with
    \begin{align*}
        d\alpha =
        -2\Re(dq\wedge \ii d\conj q).
    \end{align*}
    Here, we have applied   \(d\conj q q = d|q|^2-\conj qdq = -\conj qdq\), which holds on \(\SS^3\) where \(|q|^2 = 1\).
\end{proof}

 \begin{proof}[Proof of  \propref{prop:SO3Prequantum}]
Let \(V,W \in \mathfrak{so}(3)\), and \(\omega=(\omega_1,\omega_2,\omega_3)\), \(\nu=(\nu_1,\nu_2,\nu_3)\) be their coefficients with respect to the standard basis of \(\mathfrak{so}(3)\) as in \secref{sec:derivation_classical_formula}. For \(PV, PW \in T_P \SO(3)\) on each \(P\in \SO(3)\), we have,
\begin{align*}
    d\eta |_P(PV,PW)=-\eta |_P([PV,PW])=-\eta |_I ([V,W])=\nu_2 \omega_3 - \omega_2 \nu_3,
\end{align*}
due to the left-invariance of the vector fields \(PV, PW\) under \(\SO(3)\). Here \([\cdot,\cdot]\) denotes the Lie bracket.  Now we compute \(\pi_2 ^* \sigma\). Let us write \(W, V\) and \(P\) column-wise as \(W=(w_1 w_2 w_3), V=(v_1 v_2 v_3)\), and \(P=(p_1 p_2 p_3)\). We have,
\begin{align*}
    \pi_2^* \sigma|_P(PV,PW)
    &=\sigma|_{p_1} (d\pi_2 (PV), d\pi_2 (PW))
    =\sigma|_{p_1} (P v_1, Pw_1 )\\
    &= \sigma|_{\ii} (v_1, w_1)
    = dy \wedge dz (v_1, w_1)
    = \nu_2\omega_3 -  \omega_2 \nu_3 ,
\end{align*}
where we used the invariance of \(\sigma\) under \(\SO(3)\) and the expression of \(\sigma\) using the Cartesian coordinates. 
    \end{proof}

\section*{Acknowledgments}
The authors acknowledge anonymous referees for their reviews and insightful suggestions, and Chris Wojtan for his continuous support through discussions. The second author thanks Anna Sisak for a fruitful discussion on prequantum bundles. 
\bibliographystyle{siamplain}
\bibliography{references}

\begin{thebibliography}{10}

\bibitem{Moritz1980geodetic}
{\em {Geodetic Reference System 1980 by H. Moritz}}, Journal of Geodesy, 74
  (2000), pp.~128--162.

\bibitem{areas2012gadm}
{\em {GADM}, database of global administrative areas}, Global Administrative
  Areas,  (2012), \url{https://gadm.org/}.

\bibitem{Arvo1995stochasticSampling}
{\sc J.~Arvo}, {\em Stratified sampling of spherical triangles}, Proceedings of
  the 22nd annual conference on Computer graphics and interactive techniques,
  (1995).

\bibitem{bates_Weinstein1997lectures}
{\sc S.~Bates and A.~Weinstein}, {\em Lectures on the Geometry of
  Quantization}, Berkeley mathematics lecture notes, American Mathematical
  Society, 1997.

\bibitem{bengtsson_zyczkowski_2006}
{\sc I.~Bengtsson and K.~Zyczkowski}, {\em Geometry of Quantum States: An
  Introduction to Quantum Entanglement}, Cambridge University Press, 2006.

\bibitem{Bevis1987ComputingTA}
{\sc M.~G. Bevis and G.~Cambareri}, {\em Computing the area of a spherical
  polygon of arbitrary shape}, Mathematical Geology, 19 (1987), pp.~335--346.

\bibitem{Binysh_2018}
{\sc J.~Binysh and G.~P. Alexander}, {\em Maxwell’s theory of solid angle and
  the construction of knotted fields}, Journal of Physics A: Mathematical and
  Theoretical, 51 (2018), p.~385202.

\bibitem{Boothby_Wang_contact}
{\sc W.~M. Boothby and H.~C. Wang}, {\em On contact manifolds}, Annals of
  Mathematics, 68 (1958), pp.~721--734.

\bibitem{chernSchroedinger}
{\sc A.~Chern, F.~Kn\"{o}ppel, U.~Pinkall, P.~Schr\"{o}der, and
  S.~Wei\ss{}mann}, {\em Schr\"{o}dinger's smoke}, ACM Trans. Graph., 35
  (2016).

\bibitem{dangskul2016construction}
{\sc S.~Dangskul}, {\em Construction of seifert surfaces by differential
  geometry},  (2016).

\bibitem{feng2006area}
{\sc S.~Feng, W.~Y. Ochieng, and R.~Mautz}, {\em An area computation based
  method for raim holes assessment}, Journal of Global Positioning Systems, 5
  (2006), pp.~11--16.

\bibitem{Geiges_2008contact}
{\sc H.~Geiges}, {\em An Introduction to Contact Topology}, Cambridge Studies
  in Advanced Mathematics, Cambridge University Press, 2008.

\bibitem{Gonzalez2010Fibonacci}
{\sc A.~Gonz\'alez}, {\em Measurement of areas on a sphere using fibonacci and
  latitude–longitude lattices}, Mathematical geosciences, 42 (2010),
  pp.~49--64.

\bibitem{Hannay1998Majorana}
{\sc J.~H. Hannay}, {\em The majorana representation of polarization, and the
  berry phase of light}, Journal of Modern Optics, 45 (1998), pp.~1001--1008.

\bibitem{iwc2022implicit_filaments}
{\sc S.~Ishida, C.~Wojtan, and A.~Chern}, {\em Hidden degrees of freedom in
  implicit vortex filaments}, ACM Transactions on Graphics, 41 (2022),
  pp.~241:1--241:14.

\bibitem{Kobayashi_Kaeler}
{\sc S.~Kobayashi}, {\em {Topology of positively pinched Kaehler manifolds}},
  Tohoku Mathematical Journal, 15 (1963), pp.~121 -- 139.

\bibitem{kobayashiNomizu1963foundations}
{\sc S.~Kobayashi and K.~Nomizu}, {\em Foundations of Differential Geometry},
  no.~Bd. 1 in Foundations of Differential Geometry, Interscience Publishers,
  1963.

\bibitem{kostantBertram1970quantization}
{\sc B.~Kostant}, {\em Quantization and unitary representations}, in Lectures
  in Modern Analysis and Applications III, C.~T. Taam, ed., Berlin, Heidelberg,
  1970, Springer Berlin Heidelberg, pp.~87--208.

\bibitem{Krishna_2000}
{\sc M.~M.~G. Krishna, J.~Samuel, and S.~Sinha}, {\em Brownian motion on a
  sphere: distribution of solid angles}, Journal of Physics A: Mathematical and
  General, 33 (2000), p.~5965.

\bibitem{najdanovic2021total}
{\sc M.~S. Najdanovi{\'c}, L.~S. Velimirovi{\'c}, and S.~R. Ran{\v{c}}i{\'c}},
  {\em The total torsion of knots under second order infinitesimal bending},
  Applicable Analysis and Discrete Mathematics, 15 (2021), pp.~283--294.

\bibitem{nakahara2003geometry}
{\sc M.~Nakahara}, {\em Geometry, Topology and Physics, Second Edition},
  Graduate student series in physics, Taylor \& Francis, 2003.

\bibitem{PinkallGross2024geometry}
{\sc U.~Pinkall and O.~Gross}, {\em Differential Geometry: From Elastic Curves
  to Willmore Surfaces}, 01 2024.

\bibitem{shapere1989geometric}
{\sc A.~Shapere and F.~Wilczek}, {\em Geometric Phases In Physics}, Advanced
  Series In Mathematical Physics, World Scientific Publishing Company, 1989.

\bibitem{sommerfeld2013electrodynamics}
{\sc A.~Sommerfeld}, {\em Electrodynamics: lectures on theoretical physics,
  vol. 3}, vol.~3, Academic Press, 1952.

\bibitem{Ramamoorthi2022}
{\sc L.~Wu, G.~Cai, S.~Zhao, and R.~Ramamoorthi}, {\em Analytic spherical
  harmonic gradients for real-time rendering with many polygonal area lights},
  ACM Trans. Graph., 39 (2020).

\end{thebibliography}

\end{document}